\documentclass[12pt,reqno]{amsart}
\usepackage{amsmath,amssymb,amsfonts,amscd,latexsym,amsthm,mathrsfs}
\renewcommand\pmod[1]{\;(\operatorname{mod}#1)}
\newtheorem{theorem}{Theorem}
\newtheorem*{claim}{Claim}
\theoremstyle{remark}

\newtheorem*{acknowledgements}{Acknowledgements}
\begin{document}

\title{On three theorems of Folsom, Ono and Rhoades}

\date{12 September 2013}

\author{Wadim Zudilin}
\address{School of Mathematical and Physical Sciences,
The University of Newcastle, Callaghan NSW 2308, AUSTRALIA}
\email{wadim.zudilin@newcastle.edu.au}

\thanks{The author is supported by the Australian Research Council.}

\begin{abstract}
In his deathbed letter to G.\,H.~Hardy, Ramanujan gave a vague definition of a mock modular function:
at each root of unity its asymptotics matches the one of a modular form, though a choice of the modular
function depends on the root of unity. Recently Folsom, Ono and Rhoades have proved an elegant result
about the match for a general family related to Dyson's rank (mock theta) function and
the Andrews--Garvan crank (modular) function\,---\,the match with explicit formulae for implied $O(1)$ constants.
In this note we give another elementary proof of Ramanujan's original claim
and outline some heuristics which may be useful for obtaining a new proof
of the general Folsom--Ono--Rhoades theorem.
\end{abstract}

\maketitle

\section{Ramanujan's claim}
\label{s1}

In his deathbed letter to G.\,H.~Hardy, Ramanujan gave a vague definition of a mock modular function.
It mainly referred to a specific asymptotic behaviour of such a function at roots of unity, and Ramanujan
singled out the following illustrative example. The parameter $q$ below is always assumed to be inside the unit disc.

\begin{claim}[{Ramanujan \cite{BR95}}]
As $q$ approaches an even root of unity of order $2k$, the difference $f(q)-(-1)^kb(q)$ is absolutely bounded.
\end{claim}

Here
$$
f(q):=1+\sum_{n=1}^\infty\frac{q^{n^2}}{(1+q)^2(1+q^2)^2\dotsb(1+q^n)^2}
$$
is the mock theta function and
\begin{align*}
b(q):=(1-q)(1-q^3)(1-q^5)\dotsb\times(1-2q+2q^4-2q^9+\dotsb)
\end{align*}
is an (almost) modular form as a function of $\tau$, where $e^{2\pi i\tau}=q$.

In order to discuss and analyse the claim we introduce the standard $q$-Pochhammer notation
$$
(a;q)_n:=\prod_{j=0}^{n-1}(1-aq^j)
\quad\text{for}\;\; n=0,1,\dots,\infty;
$$
the above functions can be given then as follows:
$$
f(q)=\sum_{n=0}^\infty\frac{q^{n^2}}{(-q;q)_n^2}
\quad\text{and}\quad
b(q)=(q;q^2)_\infty\sum_{n=-\infty}^\infty(-1)^nq^{n^2}=\frac{(q;q)_\infty}{(-q;q)_\infty^2}.
$$

Recently Folsom, Ono and Rhoades proved, in two different ways, that Ramanujan's claim can be significantly refined.
Namely, they showed that the difference $f(q)-(-1)^kb(q)$ \emph{has a limit} as $q$ approaches the corresponding even root.

\begin{theorem}[{Folsom, Ono and Rhoades \cite{FOR13a,FOR13b}}]
\label{FOR1}
If $\zeta$ is a primitive even order $2k$ root of unity, then, as $q$ approaches $\zeta$ radially within the unit disc,
\begin{equation}
\lim_{q\to\zeta}\bigl(f(q)-(-1)^kb(q)\bigr)=-4u(\zeta),
\label{ex1}
\end{equation}
where
$$
u(q):=\sum_{n=0}^\infty(-q;q)_n^2q^{n+1}.
$$
\end{theorem}

\begin{theorem}[{Folsom, Ono and Rhoades \cite{FOR13b}}]
\label{FOR2}
If $\zeta$ is a primitive even order $2k$ root of unity, then, as $q$ approaches $\zeta$ radially within the unit disc,
\begin{equation}
\lim_{q\to\zeta}\bigl(f(q)-(-1)^kb(q)\bigr)
=\begin{cases}
-4\psi(-\zeta) & \text{for $k$ even}, \\
\phantom+2\phi(-\zeta) & \text{for $k$ odd},
\end{cases}
\label{ex2}
\end{equation}
where
$$
\psi(q):=\sum_{n=0}^\infty(-q^2;q^2)_nq^{n+1}
\quad\text{and}\quad
\phi(q):=1+\sum_{n=0}^\infty(-1)^n(q;q^2)_nq^{2n+1}.
$$
\end{theorem}

Note that the sums on the right-hand sides in~\eqref{ex1} and \eqref{ex2} terminate at the even root of unity.

It is not at all obvious that there are no modular forms
which \emph{exactly} cut out the singularities of a mock theta function,
in particular, of Ramanujan's $f(q)$:
the presence on the right-hand side of \eqref{ex1} (or~\eqref{ex2})
of a nonzero term, which depends on the root of unity,
is essential. This is proven Griffin, Ono and Rolen in the 2013 paper~\cite{GOR13}.

The reader intrigued by Ramanujan's mock theta functions is
referred to the inspiring expositions of Ono~\cite{Ono09} and Zagier~\cite{Za09}
on development of the subject.

Interestingly enough, Theorem~\ref{FOR2} possesses an elementary proof \cite{FOR13b} that makes use of $q$-series
transformations only, while Theorem~\ref{FOR1} is a particular instance of a much more general result
(Theorem~\ref{FOR3} stated below) whose proof uses a modern machinery of mock theta functions~\cite{FOR13a}.
The principal goal of this note is to produce a simpler proof (challenged in~\cite{FOR13b}) of Theorem~\ref{FOR1}, a proof
that Ramanujan had all ingredients for. All ingredients except possibly time.

\section{Partition generating functions}
\label{s2}

As pointed out by many authors, the relation between Dyson's rank generating function
$$
R(w;q):=\sum_{n=0}^\infty\frac{q^{n^2}}{(wq;q)_n\cdot(w^{-1}q;q)_n}
$$
and the series
$$
U(w;q):=\sum_{n=0}^\infty(wq;q)_n\cdot(w^{-1}q;q)_n\cdot q^{n+1},
$$
which is related to count of strongly unimodal sequences, has been already given by Ramanujan \cite[Entry 3.4.7, p.~67]{AB09}:
\begin{equation}
R(w;q)+(1-w)(1-w^{-1})U(w;q)
=\frac{1-w^{-1}}{(q;q)_\infty}\sum_{n=-\infty}^\infty\frac{q^{n(n+1)/2}(-w)^n}{1-w^{-1}q^n}.
\label{rama1}
\end{equation}
The left-hand side of the series is nothing but a limiting case of bilateral $_2\psi_2$-series,
$$
\sum_{n=-\infty}^\infty\frac{q^{n^2}}{(wq;q)_n\cdot(w^{-1}q;q)_n},
$$
and the above equality \eqref{rama1} follows from a general transformation due to W.\,N.~Bailey.

Though the Andrews--Garvan crank function
$$
C(w;q):=\frac{(q;q)_\infty}{(wq;q)_\infty\cdot(w^{-1}q;q)_\infty}
$$
is not immediately linked to $R(w;q)$ and $U(w;q)$, its similarity with the right-hand side of~\eqref{rama1}
becomes apparent from the expression \cite[Entry 12.2.2, p.~264]{AB05}
\begin{equation}
C(w;q)=\frac{1-w^{-1}}{(q;q)_\infty}\sum_{n=-\infty}^\infty\frac{q^{n(n+1)/2}(-1)^n}{1-w^{-1}q^n}.
\label{rama2}
\end{equation}
A surprising fact is that the asymptotics of $C(w;q)$ is related, in a simple way,
to the asymptotics of \eqref{rama1} when $w$ is chosen to be a root of unity and $q$ approaches another root
of unity radially. This remarkable relation is proven in the recent work of Folsom, Ono and Rhoades.
The notation $\zeta_m$ below is used for the root of unity $e^{2\pi i/m}$.

\begin{theorem}[{Folsom, Ono and Rhoades \cite{FOR13a,FOR13b}}]
\label{FOR3}
Let $1\le a<b$ and $1\le h<m$ be integers with $\gcd(a,b)=\gcd(h,m)=1$ and $b\mid m$.
If $h'$ is an integer satisfying $hh'\equiv1\pmod m$, then, as $q$ approaches $\zeta_m^h$ radially within the unit disc,
we have
\begin{equation}
\lim_{q\to\zeta_m^h}\bigl(R(\zeta_b^a;q)-\zeta_{b^2}^{h'a^2m}\cdot C(\zeta_b^a;q)\bigr)
=-(1-\zeta_b^a)(1-\zeta_b^{-a})U(\zeta_b^a;\zeta_m^h).
\label{for1}
\end{equation}
\end{theorem}

Taking $a=1$, $b=2$ and $m=2k$, so that $\zeta_b^a=-1$ and $\zeta=\zeta_m^h$ is a primitive even order $2k$ root of unity,
we arrive at Theorem~\ref{FOR1}, because in this case $f(q)=R(-1;q)$, $b(q)=C(-1;q)$ and $u(q)=U(-1;q)$.

\section{Proof of Theorem~\ref{FOR1}}
\label{s3}

\begin{proof}[Proof of Theorem~\textup{\ref{FOR1}}]
Since $f(q)=R(-1;q)$, $b(q)=C(-1;q)$ and $u(q)=U(-1;q)$, formulas \eqref{rama1} and \eqref{rama2} imply
\begin{align*}
f(q)+4u(q)-b(q)
&=\frac4{(q;q)_\infty}\sum_{j=-\infty}^\infty\frac{q^{j(2j-1)}}{1+q^{2j-1}},
\\
f(q)+4u(q)+b(q)
&=\frac4{(q;q)_\infty}\sum_{j=-\infty}^\infty\frac{q^{j(2j+1)}}{1+q^{2j}}.
\end{align*}
The right-hand sides can be further transformed:
\begin{subequations}
\label{tr}
\begin{align}
\frac1{(q;q)_\infty}\sum_{j=-\infty}^\infty\frac{q^{j(2j-1)}}{1+q^{2j-1}}
&=2(-q;q)_\infty^2\sum_{n=1}^\infty\frac{(-1)^{n-1}(q;q^2)_{n-1}q^{n^2}}{(-q;q^2)_n^2},
\label{tr1}
\\
\frac1{(q;q)_\infty}\sum_{j=-\infty}^\infty\frac{q^{j(2j+1)}}{1+q^{2j}}
&=\frac12\,(-q;q)_\infty^2\sum_{n=0}^\infty\frac{(-1)^n(q;q^2)_nq^{n^2}}{(-q^2;q^2)_n^2};
\label{tr2}
\end{align}
\end{subequations}
the first formula is \cite[eq.~(12.2.7), p.~264]{AB05} with $c=1$ and $q$ replaced with $-q$, while
the second one is \cite[eq.~(12.2.6), p.~264]{AB05} with $c=1$.

It remains to notice that the pre-factor $(-q;q)_\infty$ vanishes at any even root of unity, while
the sums on the right-hand sides in~\eqref{tr}
have finite limits as $q$ approaches an order $2k$ root of unity, when $k$ is even and odd, respectively.
\end{proof}

The right-hand side in~\eqref{tr1} admits a different presentation \cite[eq.~(12.5.1), p.~280]{AB05}
that leads to
$$
\frac1{(q;q)_\infty}\sum_{j=-\infty}^\infty\frac{q^{j(2j-1)}}{1+q^{2j-1}}
=2(-q;q)_\infty^2\sum_{n=1}^\infty\frac{(-1)^{n-1}(-q^2;q^2)_{n-1}q^n}{(-q;q^2)_n},
$$
a formula that can be used instead of~\eqref{tr1} in the proof. The `literal' analogue of the latter sum
for the right-hand side in~\eqref{tr2} is instead a partial theta-function \cite[eq.~(6.3.5), p.~120]{AB09},
$$
\sum_{n=0}^\infty\frac{(-1)^n(-q;q^2)_nq^n}{(-q^2;q^2)_n}
=\sum_{n=0}^\infty(-1)^nq^{n(n+1)/2}.
$$
In the odd $k$ case the radial asymptotics can be also controlled with a help of the different identity
\cite[eq.~(3.6.14), p.~79]{AB09}
\begin{align*}
\sum_{j=-\infty}^\infty\frac{q^{j(2j+1)}}{1+q^{2j}}
=(-q;q)_\infty(-q;q^2)_\infty\sum_{n=-\infty}^\infty\frac{(-q)^{n(n+1)/2}}{1+q^{2n}}.
\end{align*}
An analogue of this identity for the even $k$ case, with a proof similar to the one given in \cite[p.~79]{AB09},
seems to be inadequate for the purposes:
$$
\sum_{j=-\infty}^\infty\frac{q^{j(2j-1)}}{1+q^{2j-1}}
=\frac{2\,(-q^2;q^2)_\infty^2(q;q^2)_\infty}{(-q;q^2)_\infty(q^2;q^2)_\infty^2}\sum_{n=-\infty}^\infty\frac{(-1)^nn(-q)^{n(n+1)/2}}{1+q^{2n}}.
$$

It is worth mentioning that the two identities \eqref{tr} were used in~\cite{Mor13} to give a short proof
of another result about mock theta functions from the paper~\cite{BOR08}.

\section{Asymptotics related to Theorem~\ref{FOR3}}
\label{s4}

It would be interesting to extend the identities of the previous section to also prove
general Theorem~\ref{FOR3}. Here we briefly discuss some related asymptotics.

In light of~\eqref{rama1}, Theorem~\ref{FOR3} means in particular that the quotient of
$$
\frac{1-w^{-1}}{(q;q)_\infty}\sum_{n=-\infty}^\infty\frac{q^{n(n+1)/2}(-w)^n}{1-w^{-1}q^n}\bigg|_{w=\zeta_b^a}
$$
and $C(\zeta_b^a;q)$ tends to $\zeta_{b^2}^{h'a^2m}$ as $q$ approaches $\zeta_m^h$ radially. This fact
follows from an elementary argument reproduced below, though knowledge of the limiting behaviour of the quotient, of course,
does not imply that the difference of the two tends to~0.

To see this fact, write $q=\zeta_m^hr$, where $r\to1^-$, and split the Appell--Lerch sums (including the one in \eqref{rama2})
into $m$ subsums according to the residue $n\pmod m$:
\begin{equation*}
\begin{aligned}
\sum_{n=-\infty}^\infty\frac{q^{n(n+1)/2}(-1)^n\zeta_b^{an}}{1-\zeta_b^{-a}q^n}
&=\sum_{c=0}^{m-1}\zeta_b^{ac}\sum_{n\equiv c\pmod m}\frac{(\zeta_m^hr)^{n(n+1)/2}(-1)^n}{1-\zeta_b^{-a}\zeta_m^{hc}r^n},
\\
\sum_{n=-\infty}^\infty\frac{q^{n(n+1)/2}(-1)^n}{1-\zeta_b^{-a}q^n}
&=\sum_{c=0}^{m-1}\sum_{n\equiv c\pmod m}\frac{(\zeta_m^hr)^{n(n+1)/2}(-1)^n}{1-\zeta_b^{-a}\zeta_m^{hc}r^n}.
\end{aligned}
\end{equation*}
As $r\to1^-$, each double sum involves a single collapsing subsum that corresponds to the residue $c=c_0$ for which
$\zeta_b^{-a}\zeta_m^{hc_0}=\zeta_m^{-am/b+hc_0}=1$, so that $c_0\equiv h'am/b\pmod m$. This results in the
root of unity
$$
\zeta_b^{ac}\big|_{c=c_0}=\zeta_b^{h'a^2m/b}=\zeta_{b^2}^{h'a^2m}
$$
for the limit of the quotient as $r\to1^-$.

\section{Algebraic independence of $q$-zeta values}
\label{s5}

The elementary technique of Section~\ref{s4} about asymptotic behaviour at roots of unity was used by Pupyrev
in~\cite{Pu05} to establish some (functional) linear and algebraic independence results for the so-called $q$-zeta values
$$
\zeta_q(s):=\sum_{m=1}^\infty\frac{m^{s-1}q^m}{1-q^m}=\sum_{n=1}^\infty\sigma_{s-1}(n)q^n,
\quad\text{where}\;\;
\sigma_{s-1}(n):=\sum_{d\mid n}d^{s-1}.
$$
For even $s$, these $q$-series are related, in a simple way, to the classical Eisenstein series. In particular,
$P(q):=1-24\zeta_q(2)$ (a quasi-modular form), and $Q(q):=1+240\zeta_q(4)$ and $R(q):=1-504\zeta_q(6)$
(modular forms of weight 4 and~6) are algebraically independent over $\mathbb C(q)$,
while all other even $q$-zeta values are expressible as polynomials in $\zeta_q(4)$ and $\zeta_q(6)$.

Presumably $\zeta_q(2)$, $\zeta_q(4)$, $\zeta_q(6)$ and all odd $q$-zeta values $\zeta_q(1)$, $\zeta_q(3)$, $\dots$
are algebraically independent over $\mathbb C(q)$.
(This can be thought of as a functional $q$-analogue of the longstanding algebraic independence of $\zeta(2)=\pi^2/6$ and
all odd zeta values $\zeta(3)$, $\zeta(5)$, $\dots$\,.)

An interesting problem is to understand how `mock' the odd $q$-zeta values are,
and how helpful their mockness is for proving the expected algebraic independence.

\begin{acknowledgements}
I warmly thank Ken Ono and Ole Warnaar for conversations that led me to the elementary proof of Theorem~\ref{FOR1}.
\end{acknowledgements}


\begin{thebibliography}{99}

\bibitem{AB05}
\textsc{G.\,E.~Andrews} and \textsc{B.\,C.~Berndt},
\emph{Ramanujan's Lost Notebook, Part I}
(Springer-Verlag, New York, 2005).

\bibitem{AB09}
\textsc{G.\,E.~Andrews} and \textsc{B.\,C.~Berndt},
\emph{Ramanujan's Lost Notebook, Part II}
(Springer-Verlag, New York, 2009).

\bibitem{BR95}
\textsc{B.\,C.~Berndt} and \textsc{R.\,A.~Rankin},
\emph{Ramanujan. Letters and commentary},
History of Math. \textbf{9}
(Amer. Math. Soc., Providence, RI \& London Math. Soc., London, 1995).

\bibitem{BOR08}
\textsc{K.~Bringmann}, \textsc{K.~Ono} and \textsc{R.\,C.~Rhoades},
Eulerian series as modular forms,
\emph{J. Amer. Math. Soc.} \textbf{21} (2008), 1085--1104.

\bibitem{GOR13}
\textsc{M.~Griffin}, \textsc{K.~Ono} and \textsc{L.~Rolen},
Ramanujan's mock theta functions,
\emph{Proc. Nat. Acad. Sci. USA} \textbf{110} (2013), 5765--5768.

\bibitem{FOR13a}
\textsc{A.~Folsom}, \textsc{K.~Ono} and \textsc{R.\,C.~Rhoades},
Mock theta functions and quantum modular forms,
\emph{Forum of Math. Pi} (2013), 22~pp. (to appear).

\bibitem{FOR13b}
\textsc{A.~Folsom}, \textsc{K.~Ono} and \textsc{R.\,C.~Rhoades},
Ramanujan's radial limits,
\emph{Preprint} (2012), 12~pp.

\bibitem{Mor13}
\textsc{E.~Mortenson},
Eulerian series as modular forms revisited,
\emph{Preprint} \texttt{arXiv:\,1304.4012 [math.NT]} (2013), 6~pp.

\bibitem{Ono09}
\textsc{K.~Ono},
Unearthing the visions of a master: harmonic Maass forms and number theory,
in: \emph{Current developments in mathematics} (Proceedings of the Harvard--MIT conference, 2008), pp.~347--454
(International Press, Somerville, MA, 2009).

\bibitem{Pu05}
\textsc{Yu.\,A.~Pupyrev},
On the linear and algebraic independence of $q$-zeta values,
\emph{Math. Notes} \textbf{78} (2005), no.~3--4, 563--568.

\bibitem{Za09}
\textsc{D.~Zagier},
Ramanujan's mock theta functions and their applications (after Zwegers and Ono--Bringmann),
in: \emph{S\'eminaire Bourbaki}, Vol.~2007/2008,
\emph{Ast\'erisque} \textbf{326} (2009), Exp. no.~986, 143--164.

\end{thebibliography}
\end{document}